\documentclass{amsart}
\usepackage{amsmath,amssymb,amsfonts,enumerate,amsthm,graphicx, textcomp}
\usepackage{tikz}
\usetikzlibrary{arrows,snakes,backgrounds}
\usepackage{graphicx}
\usepackage{eucal}
\usepackage{mathrsfs}

\newcommand{\depth}{\mbox{depth}\,}

\renewcommand{\dim}{\mbox{dim}\,}

\newtheorem{thm}{Theorem}[section]
\newtheorem{conj}[thm]{Conjecture}
\newtheorem{cor}[thm]{Corollary}

\newtheorem{prop}[thm]{Proposition}

\newtheorem{exam}[thm]{Example}
\newtheorem{rem}[thm]{Remark}

\numberwithin{equation}{section}

\begin{document}
 \bibliographystyle{amsplain}

 \title[Simplicial complexes of small codimension]{Simplicial complexes of small codimension}

\author{Matteo Varbaro} \address{Dipartimento di Matematica, Universita' di Genova, Via Dodecaneso 35, Genova 16146, Italy}
\author{Rahim Zaare-Nahandi}
      \address{Rahim Zaare-Nahandi\\School of Mathematics, Statistics \&
      Computer Science, University of Tehran, Tehran, Iran.}

\thanks{R. Zaare-Nahandi was supported in part by a grant from the University of Tehran}

\thanks{Emails: varbaro@dima.unige.it, rahimzn@ut.ac.ir}



\begin{abstract}
We show that a Buchsbaum simplicial complex of small codimension must have large depth. More generally, we achieve a similar result for ${\rm CM}_t$ simplicial complexes, a notion generalizing Buchsbaum-ness, and we prove more precise results in the codimension 2 case. Along the paper, we show that the ${\rm CM}_t$ property is a topological invariant of a simplicial complex.
\end{abstract}

\maketitle

\section{Introduction}

In \cite{Ha 74}, Hartshorne proposed his tantalizing conjecture concerning smooth varieties of small codimension in some projective space. Precisely, if $R=K[x_1,\ldots ,x_n]$ is the polynomial ring in $n$ variables over a field $K$, the conjecture declaims:

\begin{conj} (Hartshorne)
If $I\subseteq R$ is a homogeneous ideal of height $h$ less than $(n-1)/3$ such that $\mathrm{Proj} R/I$ is nonsingular, then $I$ is a complete intersection.
\end{conj}

If $h=2$, then the condition $h<(n-1)/3$ is equivalent to $n>7$. In this case, by a result of Evans and Griffith \cite[Theorem 3.2]{EG 81}, the conjecture is equivalent to:

\begin{conj} 
If $I\subseteq R$ is a homogeneous ideal of height $2$ such that $\mathrm{Proj} R/I$ is nonsingular, and $n>7$, then $R/I$ is Cohen-Macaulay.
\end{conj}

The present article has no pretension to give new insights on the conjecture of Hartshorne: the only result in this direction is Corollary \ref{c:smooth}, stating that $R/I$ has depth larger than $n-2h$ if furthermore $I$ admits a square-free initial ideal. Rather, this paper brings the philosophy of the conjecture to the world of combinatorial commutative algebra, as it had already been done, to some extent, in \cite{DHS13}.

If $\Delta$ is a simplicial complex in $n$ variables, $\mathrm{Proj} K[\Delta]$ is almost never smooth, so Hartshorne's conjecture is not interesting when stated for $\mathrm{Proj} K[\Delta]$. The notion of Cohen-Macaulay-ness in codimension $t$ was introduced, independently and with the sole difference concerning a purity matter, in \cite{MNS 11} and in \cite{HYZ 12}. In \cite{MNS 11} this concept was suggested as the right one to measure the singularities of a simplicial complex: $\Delta$ is Cohen-Macaulay in codimension $t$ (according to \cite{HYZ 12}) if and only if $\Delta$ is pure of singularity dimension less than $t-1$ (according to \cite{MNS 11}). 
In particular, if $\Delta$ has negative singularity dimension, it is Buchsbaum. So, somehow Buchbaum-ness plays the role of 'smooth-ness' for simplicial complexes. This way of thinking is also supported from the results in the recent paper \cite{CV 18}, which imply that, if the ideal defining a smooth projective variety has a square-free Gr\"obner degeneration, then the associated simplicial complex is Buchsbaum. With this definition in mind, the same philosophy that led Hartshorne to make his conjecture brings one to expect the following: 
If $\Delta$ is a Buchbaum simplicial complex with small codimension, then $K[\Delta]$ should have large depth.

\vskip 2mm

In this note, we show that if $\Delta$ is a $(d-1)$-dimensional Buchbaum simplicial complex on $d+2$ vertices, then $\depth K[\Delta]\geq d-1$. Moreover, in this case $K[\Delta]$ is not Cohen-Macaulay if and only if $\Delta$ is the Alexander dual of (the clique complex of) the $(d+2)$-cycle (Proposition \ref{p:chardepth}). More generally, if $\Delta$ is a $(d-1)$-dimensional Buchsbaum simplicial complex on $n$ vertices, then $\depth K[\Delta]\geq 2d-n+1$. Even more generally, if $\Delta$ is Cohen-Macaulay in codimension $t$, then $K[\Delta]$ satisfies the condition of Serre $S_{2d-n-t+2}$ (Corollary \ref{yan}).  Along the way, we also prove that being Cohen-Macaulay in codimension $t$ is a topological invariant (Theorem \ref{t:topinv}).

\bigskip

The paper is structured as follows: a brief review of some preliminaries and conventions is given in Section 2, where the topological invariance of Cohen-Macaulay-ness in an arbitrary  codimension is also proved. Section 3 is devoted to the connection between Cohen-Macaulay-ness of a simplicial complex in some codimension with linearity of the Stanley-Reisner ideal of the Alexander dual of the simplicial complex up to a certain step. This leads to a connection between Cohen-Macaulay-ness in a certain codimension with the $S_r$ condition of Serre.   Some corollaries and relevant examples are also given. In Section 4,  the case of codimension 2 simplicial complexes is analyzed in more detail, and a combinatorial proof of the main result of Section 3 in the codimension 2 case is provided.

\section{Preliminaries and conventions}

Let $R = K[x_1,\dots, x_n]$ be the ring of polynomials over a field $K$, equipped  with the standard grading. For integers $p\ge 1$ and $d\ge 2$, we say that a simplicial complex $\Delta$ on $n$ vertices satisfies the Green-Lazarsfeld property $N_{d,p}$ if $I_{\Delta}$ is generated in degree $d$ and the first $p$ steps of the minimal graded free resolution
$$\dots \longrightarrow  F_p \xrightarrow{\varphi_p} F_{p-1}  \xrightarrow{\varphi_{p-1}} \dots  \xrightarrow{\varphi_1} F_0 \longrightarrow I_{\Delta} \longrightarrow 0$$
of $I_{\Delta}$ are linear, in the sense that $\varphi_1, \dots, \varphi_{p-1}$ are represented by matrices of linear forms.

A simplicial complex $\Delta$ is said to satisfy the Serre's condition $S_r$ if $\widetilde{H}_i({\rm link}_\Delta F; K) $ vanishes for all $F \in \Delta$ and for all $i < {\rm min} \{r-1, {\rm dim}({\rm link}_\Delta F)\}$, where $\widetilde{H}_i (\Delta; K)$ is the $i$th reduced homology group of $\Delta$ over the field $K$. This is equivalent to the usual definition of the condition $S_r$ on $K[\Delta]$.

By a ${\rm CM}_t$ simplicial complex, we mean a pure simplicial complex $\Delta$ which is Cohen-Macaulay in codimension $t$, namely a simplicial complex such that ${\rm link}_\Delta F$ is Cohen-Macaulay for all $F\in \Delta$ with $|F| \ge t$.

\begin{rem}\label{r:serre}
Let $\Delta$ be a pure simplicial complex of dimension $d-1$. It follows by the definition that $\Delta$ satisfies the $S_r$ condition $\implies$ $\Delta$ is ${\rm CM}_{d-r}$. The vice versa is false, just think to a disconnected Buchsbaum simplicial complex $\Delta$ (such a $\Delta$ is ${\rm CM}_1$ but does not even satisfy $S_2$). On the other hand, we will show in Corollary \ref{yan} that $\Delta$ is ${\rm CM}_t$ on $n$ vertices $\implies$ $\Delta$ satisfies the $S_{2d-n-t+2}$ condition.
\end{rem}

\begin{rem}
The notion of singularity dimension has been considered in \cite{MNS 11} as follows: a simplicial complex $\Delta$ has singularity dimension less than $m$ if ${\rm link}_\Delta F$ is Cohen-Macaulay  for all $F\in \Delta$ with $\dim F\geq m$ (by convention, $\dim \emptyset = -1$).
So a simplicial complex $\Delta$ is ${\rm CM_t}$ if and only if it is pure and has singularity dimension less than $t-1$.
\end{rem}

\begin{rem}
The phrase ``Cohen-Macaulay in codimension $t$'' in the present paper has a different meaning from the phrase ``Cohen-Macaulay in codimension $c$'' considered in \cite{MNS 11}. In fact, 
according to \cite[Definition 3.6]{MNS 11}, even if $\Delta$ is a pure simplicial complex of dimension $d-1$, then in \cite{MNS 11} ``$\Delta$ Cohen-Macaulay in codimension $c$'' means that ${\rm link}_\Delta F$ is Cohen-Macaulay for all $F\in \Delta$ with $|F| = d-1-c$.
\end{rem}

For an $ R$-module $M$ we write ${\rm dim}M$ for the Krull dimension of $M$; when $M=0$ we write by convention ${\rm dim}M=-\infty$.

\begin{rem}
Notice that $\Delta$ is a pure $(d-1)$-dimensional simplicial complex if and only if $${\rm dim \ Ext}_R^{n-i}(K[\Delta],R)< i \ \ \ \forall \ i< d.$$ On the other hand, it has been proved in \cite[Corollary 7.4]{MNS 11} that $\Delta$ has singularity dimension $<m$ if and only if $${\rm dim \ Ext}_R^{n-i}(K[\Delta],R)\leq m \ \ \ \ \forall \ i<d.$$ So, if $\Delta$ has singularity dimension $<m$ and ${\rm depth} K[\Delta]>m$, then $\Delta$ is pure. In particular, since ${\rm depth} K[\Delta]>0$ for any simplicial complex $\Delta$, the following are equivalent:
\begin{enumerate}
\item $\Delta$ is Buchsbaum.
\item $\Delta$ has singularity dimension $<0$.
\item $\Delta$ is ${\rm CM_1}$.
\end{enumerate}
\end{rem}

A property of a simplicial complex $\Delta$ is a topological invariant of $\Delta$ if it holds for any simplicial complex whose geometric realization is homeomorphic to the one of $\Delta$. Next we prove that the properties of satisfying $S_r$, being ${\rm CM}_t$, and having singularity dimension $<m$ are topological invariants. This fact has essentially been proved by Yanagawa in \cite{Y 11}. We report his result in our context for the convenience of the reader. We keep the same notations used in \cite{Y 11}.

\begin{thm}\label{t:topinv}
Let $\Delta$ be a $(d-1)$-dimensional simplicial complex on $n$ vertices. Then, for all $i\in\mathbb{N}$,
$${\rm dim \ Ext}_R^{n-i}(K[\Delta],R)$$
is a topological invariant of $\Delta$. In particular, satisfying $S_r$, being ${\rm CM}_t$, and having singularity dimension $<m$ are topological invariants.

\end{thm}
\begin{proof}
Let $X$ be a topological realization of $\Delta$. If ${\rm dim \ Ext}_R^{n-i}(K[\Delta],R)\leq 0$, then ${\rm dim \ Ext}_R^{n-i}(K[\Delta],R)=0$ if and only if ${\rm Ext}_R^{n-i}(K[\Delta],R)\neq 0$ if and only if $\widetilde{H}^{i-1}(X;K)\neq 0$, so we can assume that ${\rm dim \ Ext}_R^{n-i}(K[\Delta],R)>0$. 

Notice that ${\rm Ext}_R^{n-i}(K[\Delta],R)=0$ for $i>d$ or $i\leq 0$, and that ${\rm Ext}_R^{n-d}(K[\Delta],R)$ is always $d$-dimensional. Therefore we will assume that $0<i<d$.  In this situation, \cite[Theorem 4.1]{Y 11} yields that ${\rm dim \ Ext}_R^{n-i}(K[\Delta],R)-1$ is equal to the dimension of the support of the sheaf $\mathcal{H}^{-i+1}(\mathcal{D}^{\bullet}_X)$ on $X$, where $\mathcal{D}^{\bullet}_X$ is the Verdier dualizing complex of $X$ with coefficients in $K$. So we have that ${\rm dim \ Ext}_R^{n-i}(K[\Delta],R)$ is a topological invariant of $\Delta$.

For the last part, notice that being pure is obviously a topological invariant and:
\begin{enumerate}
\item $\Delta$ satisfies $S_r$ (for $r\geq 2$) $\iff$ ${\rm dim \ Ext}_R^{n-i}(K[\Delta],R)\leq i-r$ $\forall \ i<d$.
\item $\Delta$ has singularity dimension $<m$ $\iff$ ${\rm dim \ Ext}_R^{n-i}(K[\Delta],R)\leq m$ $\forall \ i<d$.
\item $\Delta$ is ${\rm CM_t}$ $\iff$ $\Delta$ is pure and ${\rm dim \ Ext}_R^{n-i}(K[\Delta],R)< t$ $\forall \ i<d$.
\end{enumerate}
\end{proof}

For further concepts and notations on simplicial complexes and combinatorial commutative algebra we refer to the standard books  \cite{St 95}, \cite{HH 11} and  \cite{MS 05}.

\section{The  ${\rm CM}_t$ property of simplicial complexes versus the Serre condition $S_r$}

In this section, for a simplicial complex $\Delta$ of dimension $d-1$ on $n$ vertices, applying a subadditivity result of Herzog and Srinivasan to the Betti diagram of the Stanley-Reisner ideal of $\Delta$, it is shown that if $\Delta$ satisfies ${\rm CM}_t$ for some $t\ge 0$, then $\Delta^\vee$ satisfies the 
$N_{n-d, 2d-n-t+2}$ condition. In other words, the minimal graded free resolution of $I_{\Delta^{\vee}}$ is linear on the first $2d-n-t+2$ steps. This leads to the implication that if $\Delta$ is $ {\rm CM}_t$ for some $t\ge 0$, then the Stanley-Reisner ring of $\Delta$ satisfies the $S_{2d-n-t+2}$ condition of Serre. \\

First we recall a generalization of the Eagon-Reiner's theorem given in \cite{HFYZ 17}. 

\begin{thm} \label{E-R} \cite[Theorem 3.1]{HFYZ 17}. Let $\Delta$ be a simplicial complex on
on $n$ vertices, $\Delta^\vee$ its Alexander dual and
$I_\Delta \subset R$ the Stanley-Reisner ideal of
$\Delta$. Then the following are equivalent:
\begin{itemize}
\item[(i)] $\Delta^\vee$ is a ${\rm CM}_t$ simplicial complex of dimension $d-1$.
\item[(ii)] $\beta_{0,j}(I_\Delta)=0 \ \ \forall \ j>n-d$ and $\beta_{i,i+j}(I_\Delta)=0 \ \ \forall \ j>n-d$ and $i+j\le n-t$.
\end{itemize}
I.e., the Betti diagram $\beta_{i,i+j}(I_\Delta)$ looks like in Figure 1.

\begin{center}
\begin{tikzpicture}[>=stealth][line width=1pt]
\draw [line width=1.pt] (-.8,0)--(7.7,0); \draw [line width=1.pt]
(0,-6.2)--(0,.8); \draw [line width=1.2pt]
(-.6,.6)--(0,0);\draw(-.2,.6) node{$i$};\draw(-.6,.3) node{$j$};
\draw(.4,.6) node{0};\draw(1.,.6) node{1};\draw(1.6,.6)
node{$\dots$};\draw(2.09,.6) node{$i$}; \draw(2.75,.6) node{\dots};
\draw(4.0,.6) node{$d-t-1$};\draw(5.35,.6)
node{$d-t$};\draw(6.15,.6) node{\dots};\draw(7.2,.6)
node{$\mbox{projdim}$};
\draw(-.85,-.4) node{$n-d$}; \draw(.4,-.4) node{$*$};\draw(1.,-.4)
node{$*$};\draw(1.6,-.4) node{$\dots$};\draw(2.12,-.4)
node{$\mbox{l. s.}$}; \draw(2.75,-.4) node{\dots}; \draw(4.0,-.4)
node{$*$};\draw(5.35,-.4) node{$*$};\draw(6.15,-.4)
node{\dots};\draw(7.2,-.4) node{$*$}; \draw[dashed](.4,-.4)
node{$*$};
\draw(-.8,-1.) node{$n-d+1$};\draw(.4,-1.) node{0};\draw(1.,-1.)
node{0};\draw(1.6,-1.) node{$\dots$};\draw(2.05,-1.) node{0};
\draw(2.75,-1.) node{\dots}; \draw(4.0,-1.) node{0};\draw(5.35,-1.)
node{$*$};
\draw(-.85,-1.6) node{$\vdots$};\draw(.4,-1.6)
node{$\vdots$};\draw(1,-1.6) node{$\vdots$};\draw(1.6,
-1.7)ellipse(3.4pt and 6pt);
\draw(-.9,-2.4) node{$ j $};\draw(.4,-2.4) node{0};\draw(1.,-2.4)
node{0}; \draw[dashed](2.15,-2.4) node{$\beta_{i,i+j}$};
\draw(-.85,-3.) node{$\vdots $};\draw(.4,-3.) node{$\vdots
$};\draw(1,-3.) node{$\vdots $};
 \draw(-.85,-3.89)node{$n-t-1$};\draw(.4,-3.89) node{0};\draw(1.,-3.89) node{0};
\draw(1.6,-3.89) node{$*$};
 \draw(-.85,-4.459)node{$n-t$};\draw(.4,-4.459) node{0};\draw(1.,-4.459) node{$*$};
 \draw(4.7,-4.459) node[scale=3.5]{$*$};
\draw(-.85,-5.) node{\vdots}; \draw(.4,-5.) node{\vdots};
\draw(1,-5.) node{\vdots};
\draw(-.85,-5.9) node{$\mbox{regularity}$};\draw(.4,-5.9) node{0};
\draw(1.,-5.9) node{$*$};
\draw [line width=.9pt,color=blue](0, -.7)--(4.8,-.7);
 \draw [line width=.9pt,color=blue](.685,-4.38)--(4.8,-.7);
   \draw [line width=.9pt,color=blue](.685, -6.2)--(.685,-4.377);
 \draw[ dotted,color=blue] (58pt,0pt) -- (58pt,-7pt);
 \draw[ dotted,color=blue] (58pt,-16pt) -- (58pt,-23pt);
\draw[ dotted,color=blue] (58pt,-36pt) -- (58pt,-60pt);
 \draw[dotted,color=blue] (0pt,-68pt) -- (47pt,-68pt);
 \draw [thick](2.5,-6.9) node {$ \mbox{Figure 1. The shape of the
Betti diagram of $I_{\Delta}$ when $\Delta^\vee$ is $\textrm{CM}_t$ }
$};
\end{tikzpicture}
\end{center}
\end{thm} 

On the other hand, Herzog and Srinivasan \cite{HS 15} proved the following ``subadditivity'' result on 
the Betti numbers of monomial ideals.

\begin{thm}\label{HS} \cite[Corollary 4]{HS 15}. Let $I = (u_1,\dots, u_m)$ be a monomial ideal of 
$R$, and let $e = {\rm max}_\ell \{ {\rm deg}(u_\ell)\}$. Then for all $j_0 \in \mathbb{Z}$:
\begin{equation}\label{eq:M}
\beta_{i,j}(I) = 0 \ \ \forall \ j> j_0 \implies \beta_{i+1,j}(I) = 0 \ \forall \ j> j_0+e. 
\end{equation}
\end{thm}

Now we prove the main result of the paper.

\begin{thm}\label{thm:main}
Let $\Delta$ be a $(d-1)$-dimensional ${\rm CM}_t$ simplicial complex on $n$ vertices. Then $\Delta^\vee$ satisfies the 
$N_{n-d, 2d-n-t+2}$ condition.
\end{thm}
\begin{proof}
Notice that $I_{\Delta^{\vee}}$ is generated in degree $n-d$. Hence the assertion is trivially valid for $2d-n-t+2\le 1$. Therefore, we may assume that $2d-n-t\ge 0$. Then, \eqref{eq:M} gives us 
\[\beta_{i,j}(I_{\Delta^{\vee}})=0 \ \ \forall \ j> j_0 \implies \beta_{i+1,j}(I_{\Delta^{\vee}}) = 0 \ \ \forall \ j> j_0+n-d.\]
By Theorem \ref{E-R}, we know that, for all $i\in\mathbb{N}$, 

\begin{equation}\label{eq:1}
\beta_{i,j}(I_{\Delta^{\vee}})=0 \ \ \forall \ i+n-d<j\leq n-t,
\end{equation}
and 

\begin{equation}\label{eq:2}
\beta_{0,j}(I_{\Delta^{\vee}})=0 \ \ \forall \ j>n-d.
\end{equation}
Now, suppose that $1\leq i\leq 2d-n-t+1$, and assume we have already proved that

\begin{equation}\label{eq:i}
\beta_{i-1,j}(I_{\Delta^{\vee}})=0 \ \ \forall \ j>i-1+n-d.
\end{equation}
By \eqref{eq:i} together with \eqref{eq:M} we have $\beta_{i,j}(I_{\Delta^{\vee}})=0$ for all $ j>i-1+2n-2d$. In particular, we have $\beta_{i,j}(I_{\Delta^{\vee}})=0$  for $i= 2d-n-t+1$, $j > (2d-n-t+1)-1+2n-2d = n-t$. On the other hand \eqref{eq:1} guarantees us that $\beta_{i,j}(I_{\Delta^{\vee}})=0$ for all $i+n-d<j\leq n-t$. Putting all together we get

\[\beta_{i,j}(I_{\Delta^{\vee}})=0 \ \ \forall \ j>i+n-d.\]
\end{proof}

In \cite{TV 15}  and, independently, in \cite{Ya 15}, the following refinement of the result of Herzog and Srinivasan is proved:

\begin{thm}\label{TorVar} \cite[Theorem 6.2, the $\mathbb{Z}$-graded part]{TV 15}. With the notation of Theorem \ref{HS}, one has:
$$\beta_{i,k}(I) = 0, \forall k=j_0,\dots, j_0+e-1 \implies \beta_{i+1,j_0+e}(I) = 0.$$
\end{thm}

This result can be applied to study the Betti numbers of $\Delta^\vee$ (inferring analog results to Theorem \ref{thm:main}) when $\Delta$ has singularity dimension less than $m$.

For $r\geq 2$, by a result of Yanagawa \cite[Corollary 3.7]{Y 00}, for a simplicial complex $\Delta$ of codimension $c$, $K[\Delta]$ satisfies the $S_r$ condition of Serre if and only if $I_{\Delta^\vee}$ satisfies the $N_{c,r}$ condition. Therefore, an interesting consequence of Theorem \ref{thm:main} is the following:

\begin{cor}\label{yan} Let $\Delta$ be a simplicial complex of dimension $d-1$ on $n$ vertices.
Assume that $\Delta$ is ${\rm CM}_t$ for some $t\ge 0$. Then $\Delta$ satisfies the $S_{2d-n-t+2}$ condition. In particular, if $\Delta$ is Buchsbaum, then ${\rm depth} K[\Delta]\geq 2d-n+1$.
\end{cor}

The following corollary is in the spirit of Hartshorne's conjecture and goes in the direction of a question raised in \cite[Question 4.2]{CV 18}.

\begin{cor}\label{c:smooth}
Let $I\subseteq R$ be a homogeneous ideal of height $h$ such that ${\rm Proj} R/I$ is nonsingular. If $I$ has a square-free initial ideal with respect to some term order, then ${\rm depth} R/I> n-2h$.
\end{cor}
\begin{proof}
Let $J$ be a square-free initial ideal of $I$. Since $R/I$ is generalized Cohen-Macaulay, $R/J$ is Buchsbaum by \cite[Corollary 2.11]{CV 18}. By Corollary \ref{yan}, then, ${\rm depth} R/J\geq n-2h+1$. We conclude since the depth cannot go up by taking the initial ideal.
\end{proof}

Another consequence, interestingly related to the result of Brehm and K\"uhnel \cite[Theorem B]{BK 87}, is the following:

\begin{cor}\label{cor:BK}
Let $\Delta$ be a $(d-1)$-dimensional Buchsbaum simplicial complex on $n$ vertices such that $\widetilde{H}_i(\Delta;K)\neq 0$ for some $i\ge 1$. Then $n\geq 2d-i$.
\end{cor}

\begin{rem}
Being the combinatorial manifolds a very special case of Buchsbaum simplicial complexes, even if the conclusion of Corollary \ref{cor:BK} is slightly weaker than the one in \cite[Theorem B]{BK 87}, it applies to a much larger class of simplicial complexes.
\end{rem}

%
%

\begin{exam}
Since Theorem \ref{thm:main} and Corollary \ref{yan} are trivial for $t \ge 2d-n+1$, it is natural to ask for examples of ${\rm CM}_t$ simplicial complexes that are not ${\rm CM_{t-1}}$ for $1\le t \le 2d-n$. Murai and Terai \cite[Example 3.5]{MT 09} considered the following simplicial complex:
 
 \begin{center}  
$\Delta=\langle \{1,2,3,5\},\{1,2,4,5\},\{1,2,4,6\},\{1,3,4,5\},\{1,3,4,6\},\{1,3,5,6\},$\\ $\{2,3,4,6\},\{2,3,5,6\},\{2,4,5,6\}\rangle,$
\end{center}
where $\Delta$ satisfies $S_3$ but is not Cohen-Macaulay. Thus $\Delta$ is Buchsbaum and the condition $1\le t \le 2d-n$ is satisfied.   
Now if $v$ is a new vertex, by \cite[Theorem 3.1 (ii)]{HYZ 15}, the cone on 
$\Delta$ with vertex $v$ is ${\rm CM}_2$ but not Buchsbaum, and again we have $1\le t \le 2d-n$. Taking further cones, one gets a family of 
${\rm CM}_t$ simplicial complexes  which are not ${\rm CM}_{t-1}$ and we have $1\le t \le 2d-n$.
\end{exam}

\begin{rem}\label{restriction}
Often, the minimal number of vertices necessary for triangulating a given $(d-1)$-dimensional combinatorial manifolds is more than $2d$. 
An exception is  an $8$ dimensional combinatorial manifold, the so called ``Brehm and K\"uhnel manifold'', which has $6$ combinatorially different triangulations on $15$ vertices (see \cite{BK 87}, \cite[Proposition 48]{L 05} and \cite{L 99}). 
\end{rem}

\section{The ${\rm CM}_t$ property and minimal chord-less cycles of graphs}

In this section, we focus on pure $(d-1)$-dimensional simplicial complexes on $d+2$ vertices, i.e. pure codimension two simplicial complexes. If $\Delta$ is such a simplicial complex, then its Alexander dual is flag, i.e., $\Delta^\vee$ is the clique complex of a graph $G$. In general, the clique complex and the independence complex of a graph $H$ will be denoted by $\Delta(H)$ and $\Delta_H$, respectively. Also, by $\overline{H}$ we will denote the complementary graph of $H$. 

\begin{thm}\label{topin}
Let $\Delta$ be a pure $(d-1)$-dimensional codimension two simplicial complex. Then the following are equivalent:
\begin{itemize}
\item[(i)] $\Delta$ is ${\rm CM}_t$,
\item[(ii)] $\Delta^\vee$ satisfies the $N_{2,d-t}$ condition, 
\item[(iii)] $\Delta$ satisfies the $S_{d-t}$ condition,
\item[(iv)] Every cycle of the 1-skeleton $G$ of $\Delta^{\vee}$ of length at most $d-t+2$ has a chord.
\end{itemize}
\end{thm}

\begin{proof}
The equivalence of (i), (ii) and (iii) is simply an application of Theorem \ref{thm:main}, Corollary \ref{yan} in the case  $n=d+2$, and Remark \ref{r:serre}. The equivalence of (ii) and (iv) follows by \cite[Theorem 2.1]{EGHP 05}. 
\end{proof}

\begin{prop}\label{p:chardepth}
If $\Delta$ is a codimension two Buchsbaum simplicial complex, then ${\rm depth K[\Delta]}\geq {\rm dim} \Delta$. Furthermore, $K[\Delta]$ is Cohen-Macaulay if and only if the 1-skeleton $G$ of $\Delta^{\vee}$ is not the $(d+2)$-cycle.
\end{prop}
\begin{proof}
Notice that $\Delta$ being Buchsbaum is equivalent to $\Delta$ being ${\rm CM_1}$. So the first part of the statement follows by Theorem \ref{topin}. If $K[\Delta]$ is not ${\rm CM_0}$, again Theorem \ref{topin} implies that $G$ has an induced chord-less $(d+2)$-cycle (in those notations, so $d={\rm dim} \Delta+1$). Since the number of vertices is $d+2$, $G$ is actually the $(d+2)$-cycle.
\end{proof}

\begin{rem}
In particular, if $\Delta$ is a codimension two  Buchsbaum simplicial complex which is not Cohen-Macaulay, then ${\rm projdim} K[\Delta] = 3$. One might expect that, in general, if $\Delta$ is a codimension 2 simplicial complex which is \ ${\rm CM}_t$ but not ${\rm CM}_{t-1}$, then ${\rm projdim} K[\Delta] = t+2$. This is false: a simple example is the Alexander dual of $\Delta = \langle \{1,2\},\{2,3\},\{3,4\},\{4,5\},\{1,5\},\{5,6\} \rangle$ which has dimension $d-1$ where $d=4=6-2=n-2$. Then $\Delta$ is  ${\rm CM}_2$ but not ${\rm CM}_1$. Nevertheless, the projective dimension of the Stanley-Reisner ring of $\Delta$ is $3$.
\end{rem}

For the sake of documenting a different method, we give an alternative proof, more combinatorial, for the equivalence of  (i) and (iv) in Theorem \ref{topin}. 
\\

\begin{thm}\label{main2} Let $G$ be a simple graph on $[n] =\{1,\dots, n\}$ with no isolated vertices. 
Let $\Delta = \Delta(G)$ be the clique complex of $G$ .
Let $r\ge 3$ be an integer. Then $\Delta^\vee$ is ${\rm CM}_{n-r}$ if an only if every cycle of $G$ 
of length at most $r$ has a chord.
\end{thm}

\begin{proof}
The ``if'' direction follows by \cite[Theorem 2.1]{EGHP 05}, \cite[Corollary 3.7]{Y 00} and Remark \ref{r:serre}, so we focus on the ``only if'' part.

Assume that $\Delta^\vee$ is ${\rm CM}_{n-r}$. We prove by induction on $r$ that every cycle of $G$ 
of length at most $r$ has a chord. The first case $r=3$  is trivial. 
Assume that $r\ge 4$. Since 
$\Delta^\vee$ is ${\rm CM}_{n-r+1}$, every cycle of $G$ of length at most $r-1$ 
has a chord. So it is enough to show that $G$ has no chord-less $r$-cycles. Assume that, on the contrary,   
$G$ has a chord-less $r$-cycle $C$.  Let
$V(C)=\{v_1, \dots, v_r\}$ and $E(C)=\{\{v_1,v_2\}, \dots, \{v_{r-1},v_r\}, \{v_r,v_1\}\}$ be the vertex set and the edge set 
of $C$, respectively. Then the induced subgraph of $\overline{G}$ 
on $V(C)$ is the graph $K_r\setminus E(C)$, where $K_r$ is the complete graph on $V(C)$. Clearly, 
$K_r\setminus E(C)$ has ${r \choose 2}  - r = r(r-3)/2$ edges. Let $F$ be the simplex on $V(G)\setminus V(C)$. Then,
$|F| = n-r$ and $F$ is a face of $\Delta^\vee$ because $V(C)\notin \Delta$. 
Thus $\Gamma = {\rm link}_{\Delta^{\vee}}F$ should be Cohen-Macaulay. We prove that this is not the case.  
Observe that the only facets of $\Delta^{\vee}$ which contain $F$ are 
$F\cup (V(C)\setminus \{v_i, v_j\})$ for some  $\{v_i, v_j\}\in \overline{C}$. 
Therefore, 
$$\Gamma = {\rm link}_{\Delta^{\vee}}F = \langle V(C)\setminus \{v_i, v_j\}: \{v_i, v_j\}\in \overline{C} \rangle.$$
In particular, ${\rm dim} \Gamma = r-3$.
We determine $h_{r-2}$ by computing the $f$-vector of $\Gamma$: to this purpose, notice that 
every subset of the vertex set of $\Gamma$ of cardinality $\le r-3$ is also a face of $\Gamma$. To see this, let $E = V(C) \setminus \{v_i,v_j,v_k\}$ 
be a subset of the vertex set of $\Gamma$ of cardinality $r-3$. Choose $1\le l \le r$ such that $l \notin \{ i, j, k\}$. Then at least one of the pairs $(i,l)$, 
$(j,l)$ and $(k,l)$ will be a non-consecutive pair modulo $r$. Let $(i,l)$ be such a pair. Then, $\{i,l\}\in \overline{G}$, and hence, $E\subset V(C)\setminus \{v_i, v_j\}$, i.e., 
$E$ is a face of $\Gamma$. Therefore we got:
\begin{center}
$f_{-1} =1, f_i = {r \choose i+1}, i = 0, \dots, r-4$ and $f_{r-3} = r(r-3)/2$.  
\end{center}
Consequently, 
$$h_{r-2} = \sum_{i=0}^{r-2} (-1)^{r-2-i} f_{i-1} = (\sum_{i=0}^{r-3} (-1)^{r-i} {r \choose i}) + r(r-3)/2 =$$
$$(1-1)^r + {r \choose r-1} - {r \choose r-2} -1 +  r(r-3)/2  = -1.$$
Hence $\Gamma$ is not Cohen-Macaulay. This completes the proof. 
\end{proof}



\begin{cor}\label{linear}
With the assumptions of Theorem \ref{main2}, assume that $G$ is 
$r$-chordal, i.e., it has no chord-less cycles of length greater than $r$. Then ${\Delta^{\vee}}$ is ${\rm CM}_{n-r}$ 
if and only if $I_\Delta = I(\overline{G})$ has a linear resolution.
\end{cor}
\begin{proof}
The assertion follows by Theorems \ref{topin} and \ref{main2} and Fr\"oberg's result that $I_\Delta = I(\overline{G})$ 
has a linear resolution if and only if $G$ is chordal \cite{Fr 88}.
\end{proof}

\begin{rem}\label{bicor}
It is easy to see that if $G$ is a bipartite graph or a chordal graph, then $\overline{G}$ can only have chord-less four cycles
(e.g., see \cite[Lemma 4.2 and Lemma 4.6 ]{HFYZ 17}). 
Assume that $G$ is a graph on $n$ vertices 
which is either bipartite or chordal. If the Alexander dual of $\Delta(\overline{G})=\Delta_G$ is ${\rm CM}_{n-4}$,  then by Corollary \ref{linear},
 $I(G)$ has a linear resolution.
\end{rem}


\section*{Acknowledgments}
This work was carried out when the second author was visiting Department of Mathematics of University of Genova,  23 October 2017 -- 3 February 2018. He is grateful to Aldo Conca for arranging this visit and for the warm hospitality, and he thanks INdAM for the partial support. The visit of this author was provided by University of Tehran as a research leave for which he is thankful to the authorities of this university.

\end{document}